\documentclass[11pt]{article}
\usepackage{amsmath,amssymb,amsthm,mathtools}
\usepackage{hyperref}
\usepackage{enumitem}
\usepackage{geometry}
\theoremstyle{plain}
\newtheorem{theorem}{Theorem}[section]
\newtheorem{proposition}[theorem]{Proposition}
\newtheorem{lemma}[theorem]{Lemma}
\newtheorem{corollary}[theorem]{Corollary}
\theoremstyle{definition}
\newtheorem{definition}[theorem]{Definition}
\theoremstyle{remark}
\newtheorem{remark}[theorem]{Remark}

\newcommand{\F}{\mathbb{F}}
\newcommand{\E}{\mathbb{E}}
\newcommand{\TV}{\mathrm{TV}}
\newcommand{\rank}{\mathrm{rank}}
\newcommand{\Sym}{\mathrm{Sym}}
\newcommand{\inv}{\mathrm{inv}}

\newcommand{\logtwo}{\log_2}
\newcommand{\mix}{\mathrm{mix}}
\newcommand{\rad}{\mathrm{rad}}

\title{Triangular cutoff threshold for the inversion walk on tournaments\\ and the state space of restricted inversions}

\author{
    Jiangdong Ai\thanks{School of Mathematical Sciences and LPMC, Nankai University.
    {\tt jd@nankai.edu.cn}. Funded by the National Natural Science Foundation of China
    (No.12522117, No.12401456), the Natural Science Foundation of Tianjin
    (No.24JCQNJC01960) and Fundamental and Interdisciplinary Disciplines Breakthrough
    Plan of the Ministry of Education of China (JYB2025XDXM207).}
}

\date{}

\begin{document}

\maketitle

\begin{abstract}
Given a labelled tournament on $[n]$, \emph{inverting} a vertex subset $X$ means reversing
every edge with both endpoints in $X$.
Alon, Powierski, Savery, Scott, and Wilmer~\cite{AlonPowierskiSaveryScottWilmer2024} asked
for the mixing time of the Markov chain that repeatedly inverts a uniformly random subset
of $[n]$.
We show that this \emph{inversion walk} has a triangular cutoff threshold.  Let $T(s):=\binom{s+1}{2}$.  For every integer sequence $s_n\in[0,n]$,
$$
  d_n(n-s_n)\longrightarrow
  \begin{cases}
    1,& T(s_n)-n\to+\infty,\\
    0,& T(s_n)-n\to-\infty.
  \end{cases}
$$
Consequently, for every fixed $\varepsilon\in(0,1)$, $
  t_{\mix}^{(n)}(\varepsilon)=n-\sqrt{2n}+O(1).$
  
We also prove quantitative one-sided bounds, with absolute constants
$C_\star<0.36$ and $\kappa<3.47$,
$$
  d_n(n+c)\le C_\star 2^{-c}\quad(c\ge0),
  \qquad
  d_n(n-s)\ge 1-\kappa\,2^{\,n-\binom{s+1}{2}}\quad(0\le s\le n).
$$

As a second result, we characterise the state space of the \emph{$k$-restricted inversion walk},
which inverts a uniformly random $k$-subset at each step.
For $n\ge 4$ and $2\le k\le n-2$, the reachable states form a coset of a subgroup
$H_k\le\F_2^{\binom{n}{2}}$ whose defining parity constraints are determined by $k\bmod 4$; equivalently, its codimension is $0,1,n-1$, or $n$ according as $k\equiv2,0,3$, or $1\pmod4$.
\end{abstract}

\section{Introduction}

A \emph{tournament} on the vertex set $[n]=\{1,\dots,n\}$ is an orientation of the complete
graph $K_n$, for each unordered pair $\{i,j\}$ exactly one of $i\to j$ or $j\to i$ is present.
There are $2^m$ labelled tournaments, where $m=\binom{n}{2}$.
For a subset $X\subseteq[n]$, \emph{inverting $X$} means reversing all arcs with both
endpoints in $X$.
A single inversion can flip up to $\binom{|X|}{2}$ edges, so the operation is simultaneously
local (on small sets) and global (affecting $\Theta(n^2)$ edges for typical $X$).

Alon, Powierski, Savery, Scott, and Wilmer~\cite{AlonPowierskiSaveryScottWilmer2024}
initiated a systematic study of inversion distance in digraphs and tournaments.
Among other results, they proved that any labelled tournament on $[n]$ can be reached from
any other by at most $(1+o(1))n$ inversions, establishing that the \emph{inversion diameter}
is $\Theta(n)$.
They also showed that a uniformly random symmetric $\F_2$-matrix has rank $n-O(\log n)$ with
high probability (Lemma~12 of~\cite{AlonPowierskiSaveryScottWilmer2024}), and asked for the
mixing time of the natural random process that repeatedly inverts a uniformly random subset
of vertices.

The latter question fits into the rich theory of random walks on groups.
Encoding tournaments relative to a fixed reference identifies the state space with the abelian
group $G=\F_2^m$, so the inversion walk is a \emph{Cayley walk}, at each step, a random group
element is drawn from a symmetric generating multiset and added to the current state.
For Cayley walks on abelian groups, Fourier analysis on $G$ reduces mixing to an estimate of
spectral gaps, which in turn reduce to character sums.
The distinctive feature here is that the generating distribution arises from quadratic functions on $\F_2^n$ (one for each clique), giving rise to Gauss-type character sums over
$\F_2$.
This connects the mixing problem to the theory of quadratic forms over $\F_2$ and the rank
of alternating bilinear forms.

Note that the state space has size $2^m=2^{\Theta(n^2)}$,
so a naive random walk that moves one edge at a time (the lazy version of the $k=2$ case;
see Proposition~\ref{hypercube}) mixes in $\Theta(m\log m)=\Theta(n^2\log n)$ steps.
By contrast, the full inversion walk mixes in $\Theta(n)$ steps, achieving a substantial polynomial-factor
speed-up by exploiting the high-dimensional structure of clique-induced inversions.

The \emph{inversion walk} $W_n$ is the Markov chain on labelled
tournaments on $[n]$ defined by, from the current tournament $T$, choose $X\subseteq[n]$
uniformly among all $2^n$ subsets and invert $X$.
(When $|X|\le 1$ the step is the identity, so the chain is automatically aperiodic.)
The stationary distribution is uniform on the $2^m$ tournaments; write $\pi$ for this
distribution and
$$
  d_n(t):=\max_{T_0}\|\mathcal{L}(W_n(t)\mid W_n(0)=T_0)-\pi\|_{\TV}
$$
for the worst-case total-variation distance at time $t$.

Let $C_\star:=\frac12\left(\sum_{j\ge1}2^{-j(2j-1)}\right)^{1/2}<0.36$,
 and $\kappa:=\prod_{j\ge1}(1-2^{-j})^{-1}<3.47.$

\begin{theorem}\label{mainthm}
Let $d_n(t)$ be the worst-case total-variation distance of the inversion walk on tournaments
on $[n]$, and let $T(s):=\binom{s+1}{2}$.
\begin{enumerate}[label=\textup{(\roman*)},leftmargin=2.6em]
\item For every integer $s\in\{0,1,\dots,n\}$, $d_n(n-s) \ge 1-\kappa\,2^{\,n-T(s)}.$

Consequently, if $s_n\in\{0,1,\dots,n\}$ and $T(s_n)-n\to+\infty$, then $d_n(n-s_n)\longrightarrow 1.$
\item If $s_n\in\{0,1,\dots,n\}$ and
$T(s_n)-n\to-\infty$, then
$d_n(n-s_n)\longrightarrow 0.
$
\item For every integer $c\ge0$,
$d_n(n+c) \le C_\star\,2^{-c}.$
\end{enumerate}
In particular, for every fixed $a\ge0$ with $a\ne\sqrt2$,
$$
  d_n\!\left(n-\lfloor a\sqrt n\rfloor\right)\longrightarrow
  \begin{cases}
    1,& a>\sqrt2,\\
    0,& 0\le a<\sqrt2.
  \end{cases}
$$
Moreover, if
$$
  \sigma_n:=\max\{s:\ T(s)\le n\}
  =\left\lfloor\frac{\sqrt{8n+1}-1}{2}\right\rfloor,
$$
then, for every fixed $\varepsilon\in(0,1)$,
$$
  n-\sigma_n-1\ \le\ t_{\mix}^{(n)}(\varepsilon)\ \le\ n-\sigma_n+1
$$
for all sufficiently large $n$.  In particular,
$$
  t_{\mix}^{(n)}(\varepsilon)=n-\sigma_n+O(1)=n-\sqrt{2n}+O(1).
$$
Furthermore, in the usual cutoff-window sense, the chain has cutoff at times $n-\sigma_n$
with every window $w_n$ satisfying $w_n\to\infty$ and $w_n=o(n)$.
\end{theorem}

\begin{remark}
\label{rmk:window}
The lower bound in item~\textup{(i)} is a support obstruction coming from inversion balls and
a sharp rank-counting tail for random symmetric matrices over $\F_2$; importantly, it is
peaked at the triangular threshold $T(s)=\binom{s+1}{2}\approx n$.  Item~\textup{(ii)} is
the matching spectral ingredient, it counts only the non-zero Fourier eigenvalues, and uses
sharp orthogonal-group denominators in the orbit count.  Together these estimates improve the
cutoff location from the $\sqrt n$-scale statement $n-\sqrt{2n}+o(\sqrt n)$ to bounded
distance from the integer triangular threshold $n-\sigma_n$.
\end{remark}

\begin{remark}
The inversion diameter of $\Theta(n)$ from~\cite{AlonPowierskiSaveryScottWilmer2024}
matches the mixing time up to constants, but the two results are logically independent.
The diameter bounds the support of $\mu_t$, while mixing concerns
the \emph{distribution} of the chain.
The connection is made precise via the inversion-ball volume argument in Section~\ref{s:balls}.
\end{remark}

Fix $k\in\{0,1,\dots,n\}$ and consider the \emph{$k$-restricted inversion walk} $W_{n,k}$,
which inverts a uniformly random $k$-subset at each step.
In the group encoding, $W_{n,k}$ moves by adding the vectors $v_X$ with $|X|=k$, so it is
irreducible on cosets of the subgroup $H_k:=\langle v_X:\ |X|=k\rangle\le \F_2^m$.
Determining $H_k$ is the first structural step toward studying the mixing of $W_{n,k}$.

\begin{theorem}
\label{Hk-thm}
Assume $n\ge 4$ and $2\le k\le n-2$.
Define the degree-parity map $\partial:\F_2^m\to\F_2^n$ by
$\partial(F)_v=\deg_F(v)\pmod 2$, and the edge-count parity $e:\F_2^m\to\F_2$ by
$e(F)=|F|\pmod 2$.
Then
$$
H_k \;=\;
\begin{cases}
\F_2^{m}, & k\equiv 2\pmod 4,\\
\ker(e), & k\equiv 0\pmod 4,\\
\ker(\partial), & k\equiv 3\pmod 4,\\
\ker(\partial)\cap\ker(e), & k\equiv 1\pmod 4.
\end{cases}
$$
\end{theorem}

The proof combines elementary parity obstructions with a dimension computation using Wilson's
diagonal form for inclusion matrices~\cite{Wilson1990}.

The paper is organized as follows: In Section~\ref{s:setup}, we set up the group encoding and the Fourier $L^2$--TV bound.
In Section~\ref{s:spectral}, we prove the spectral estimates, including the matching triangular upper
bound on the lower side of the transition.
In Section~\ref{s:balls}, we establish the lower side via inversion balls and a sharp rank-counting bound.
In Section~\ref{s:pf-cutoff}, we combine the two estimates to obtain the triangular cutoff threshold and the bounded localization of the mixing time.
In Section~\ref{s:restr}, we prove Theorem~\ref{Hk-thm}, treat boundary cases, and show the
classical hypercube asymptotics for $k=2$.
In Section~\ref{s:open}, we collect open problems.

\section{Preliminaries and group encoding}
\label{s:setup}

Fix a reference tournament $T_{\mathrm{ref}}$ on $[n]$ (the specific choice does not matter).
Let $m=\binom{n}{2}$ and index coordinates of $\F_2^m$ by the $m$ unordered pairs
$\{i,j\}\subseteq[n]$.
Encode each tournament $T$ by $z(T)\in\F_2^{m}$, where the coordinate indexed by $\{i,j\}$
equals $1$ iff $T$ disagrees with $T_{\mathrm{ref}}$ on the orientation of $\{i,j\}$.
The map $T\mapsto z(T)$ is a bijection from the set of tournaments to $G=\F_2^m$.

For $X\subseteq[n]$, define $v_X\in\F_2^m$ by
$$
  (v_X)_{\{i,j\}}=\begin{cases}
  1,& \{i,j\}\subseteq X,\\
  0,& \text{otherwise}.
  \end{cases}
$$
Since inverting $X$ flips exactly the edges of the induced subgraph $K_n[X]$,
one checks that $z(\mathrm{inv}_X(T))=z(T)+v_X$ in $\F_2^m$.
Thus \emph{$W_n$ is a Cayley walk on $G=\F_2^m$} driven by the step distribution
$\nu=\mathrm{Unif}\{v_X : X\subseteq [n]\}$ (with multiplicity, since
$v_\emptyset=v_{\{i\}}=0$).
The support of $\nu$ generates $G$, because the steps with $|X|=2$ are exactly the
standard basis vectors of $\F_2^m$ indexed by the edges of $K_n$.  Hence the walk is
irreducible.  The stationary distribution of any irreducible Cayley walk on a finite abelian
group is uniform, so, here $\pi=\mathrm{Unif}(G)$.

Characters of $G=\F_2^m$ are $\chi_A(z)=(-1)^{\langle A,z\rangle}$ for $A\in\F_2^m$,
where $\langle\cdot,\cdot\rangle$ is the standard inner product over $\F_2$.
Since each character $\chi_A$ is a group homomorphism, a direct computation gives
$P\chi_A = \lambda_A \chi_A$, so the characters diagonalise the transition operator.
The eigenvalue corresponding to $\chi_A$ is
$$
  \lambda_A=\E_{X\sim\mathrm{Unif}(2^{[n]})}\!\big[(-1)^{\langle A,v_X\rangle}\big],\quad A\in\F_2^m,
$$
with $\lambda_0=1$ for the trivial character.
For the chain started at a fixed state $T_0$, the distribution at time $t$ satisfies
\begin{equation}
\label{fourierexp}
  \mu_t(\{T\})=\frac{1}{|G|}\sum_{A\in\F_2^m}\lambda_A^t\,(-1)^{\langle A, z(T)\rangle}(-1)^{\langle A, z(T_0)\rangle}.
\end{equation}
The standard $L^2$--TV inequality (see e.g.~\cite[Proposition~7.14]{LevinPeres2017}) then gives
\begin{equation}
\label{L2bound}
  \|\mu_t-\pi\|_{\TV}\ \le\ \frac12\sqrt{\sum_{A\neq 0}\lambda_A^{2t}}.
\end{equation}
Note that the right-hand side is independent of the starting state $T_0$, so \eqref{L2bound}
bounds the worst-case distance $d_n(t)$ directly.

\begin{definition}
\label{def:inv}
For tournaments $T,T'$ on $[n]$, the \emph{inversion distance} $\inv(T,T')$ is the minimum
number of inversions needed to transform $T$ into $T'$.
The \emph{inversion ball} of radius $t$ around $T$ is
$B_t(T):=\{T':\ \inv(T,T')\le t\}$.
\end{definition}

Note that $\inv(T,T')$ equals the minimum number of clique vectors $v_{X_1},\dots,v_{X_s}$
(repetitions allowed) with $z(T')-z(T)=\sum v_{X_i}$ in $G$.

In particular, after $t$ steps started from $T_0$, the chain $W_n$ is supported on $B_t(T_0)$.

\begin{proposition}
\label{gram-representation}
Let $X_1,\dots,X_t$ be the random subsets chosen in the first $t$ steps, and let
$M\in\F_2^{t\times n}$ be the matrix whose $\ell$-th row is the indicator vector
$\mathbf 1_{X_\ell}$.  Then, relative to the starting tournament,
$$
  z(W_n(t))-z(W_n(0))
  =\operatorname{offdiag}(M^\top M),
$$
where the right-hand side denotes the strict upper-triangular part of the symmetric Gram
matrix $M^\top M$ over $\F_2$.
\end{proposition}

\begin{proof}
For an edge $\{i,j\}$, the $(i,j)$ entry of $M^\top M$ is
$\sum_{\ell=1}^t M_{\ell i}M_{\ell j}$, computed in $\F_2$.  This is exactly the parity of the
number of chosen sets $X_\ell$ containing both $i$ and $j$, i.e. the parity of the number of
times the edge $\{i,j\}$ has been flipped.
\end{proof}

\section{Spectral estimates and the triangular upper bound}
\label{s:spectral}

We compute the eigenvalues $\lambda_A$ in terms of the rank of a quadratic form.

Identify $A\in\F_2^m$ with an undirected graph $H_A$ on $[n]$ whose edge-set is the support
of $A$.  For $x\in\F_2^n$ (the indicator vector of a set $X\subseteq[n]$) define
$$
  q_A(x):=\sum_{\{i,j\}\in E(H_A)} x_i x_j \ \in\ \F_2.
$$
Since $|E(H_A[X])|\equiv q_A(x)\pmod 2$, we have $\langle A,v_X\rangle=q_A(x)$, and therefore
\begin{equation}
\label{walsh-eig}
  \lambda_A
  =2^{-n}\sum_{x\in\F_2^n}(-1)^{q_A(x)}.
\end{equation}
Thus each eigenvalue is a (normalised) Walsh--Hadamard sum of the quadratic Boolean function
$q_A:\F_2^n\to\F_2$.

The \emph{polarisation} of a quadratic function $q:\F_2^n\to\F_2$ is
$$
  B_q(x,y):=q(x+y)+q(x)+q(y)+q(0),\quad x,y\in\F_2^n.
$$
If $q$ is quadratic, then $B_q$ is a symmetric bilinear form over $\F_2$.
Moreover, for any $x\in\F_2^n$ we have $B_q(x,x)=q(0)+q(x)+q(x)+q(0)=0$,
so $B_q$ is \emph{alternating}.
Write $r(q):=\rank(B_q)$ for the rank of this alternating form, and $r(A):=r(q_A)$.
Since $B_q$ is alternating, $r(q)$ is always even.

\begin{lemma}
\label{qwbound}
For every quadratic $q:\F_2^n\to\F_2$, $
  \Big|\sum_{x\in\F_2^n}(-1)^{q(x)}\Big|\ \le\ 2^{\,n-r(q)/2}.
$
Consequently, the eigenvalues of the inversion walk satisfy $|\lambda_A|\le 2^{-r(A)/2}$.
\end{lemma}

\begin{proof}
Let $S=\sum_{x\in\F_2^n}(-1)^{q(x)}$.  Squaring and substituting $z=x+y$:
$$
  S^2=\sum_{x,y\in\F_2^n}(-1)^{q(x)+q(y)}
  =\sum_{x,z\in\F_2^n}(-1)^{q(x)+q(x+z)}.
$$
Using the definition of $B_q$, we have $q(x)+q(x+z)=B_q(x,z)+q(z)+q(0)$, so
$$
  S^2
  =(-1)^{q(0)}\sum_{z\in\F_2^n}(-1)^{q(z)}\sum_{x\in\F_2^n}(-1)^{B_q(x,z)}.
$$
The inner sum $\sum_x(-1)^{B_q(x,z)}$ is a character sum over a linear map;
it equals $2^n$ if $z\in\rad(B_q):=\{z: B_q(x,z)=0\ \forall x\}$,
and equals $0$ otherwise (since when $z\notin\rad(B_q)$ the map $x\mapsto B_q(x,z)$ is a
non-trivial linear functional, whose $\pm 1$ values cancel perfectly).  Therefore
$$
  |S|^2\le |\rad(B_q)|\cdot 2^n = 2^{n-r(q)}\cdot 2^n = 2^{2n-r(q)}.
$$
Taking square roots: $|S|\le 2^{n-r(q)/2}$.  The eigenvalue bound follows from~\eqref{walsh-eig}.
\end{proof}

The same computation gives the following standard dichotomy, which we record explicitly
because it is the key point needed to sharpen the window.  If $R_A=\rad(B_{q_A})$ and
$S_A=\sum_x(-1)^{q_A(x)}$, then the argument above gives $S_A^2=2^n\sum_{z\in R_A}(-1)^{q_A(z)}.$
The restriction $q_A|_{R_A}$ is linear, since its polar form is identically zero on $R_A$.
Hence the last sum is $0$ unless $q_A|_{R_A}\equiv0$, and is $|R_A|=2^{n-r(A)}$ otherwise.
Consequently,
$\lambda_A=0$ unless $q_A|_{R_A}\equiv0$,
 and otherwise $|\lambda_A|=2^{-r(A)/2}.$
The next lemma exploits this vanishing.

\begin{lemma}\label{nonzero-count}
For even $r\in\{2,4,\dots,n\}$, let
$N_{n,r}:=\big|\{A\in\F_2^{\binom n2}:\ r(A)=r,\ \lambda_A\ne0\}\big|$.
Then
$$
  N_{n,r}
  \le
  2^{\,n(r-1)+n\logtwo(1+2^{-r/2})-r(r-1)/2+1}.
$$
\end{lemma}

\begin{proof}
Let $V=\F_2^n$.  Suppose $r(A)=r$ and $\lambda_A\ne0$.  Let
$R=\rad(B_{q_A})$.  Since $q_A|_R\equiv0$, the quadratic form $q_A$ descends to a
non-degenerate quadratic form $\bar q$ on $U:=V/R$, with polar form $\bar B$.
Writing $u_i=e_i+R\in U$, the vectors $u_1,\dots,u_n$ span $U$, satisfy
$\bar q(u_i)=q_A(e_i)=0$, and determine the graph $A$ through
$ A_{ij}=\bar B(u_i,u_j)$ where $1\le i<j\le n.$
Conversely, any spanning $n$-tuple of singular vectors in a non-degenerate quadratic space
$(U,\bar q)$ of dimension $r$ gives such a graph. If $\phi:\F_2^n\to U$ sends $e_i$ to
$u_i$, then the resulting quadratic form is $\bar q\circ\phi$, its polar form has rank $r$,
and it vanishes on its radical $\ker\phi$.

For fixed isometry type, two spanning singular tuples give the same graph precisely when
they differ by an isometry.  Indeed, suppose $(u_i)$ and $(u'_i)$ have the same pairings
$\bar B(u_i,u_j)=\bar B(u'_i,u'_j)$.  If $\sum_i\alpha_i u_i=0$, then
$\sum_i\alpha_i u'_i$ pairs trivially with every $u'_j$; since the $u'_j$ span $U$ and
$\bar B$ is non-degenerate, $\sum_i\alpha_i u'_i=0$.  Thus the rule
$\sum_i\alpha_i u_i\mapsto\sum_i\alpha_i u'_i$ is well-defined.  It preserves $\bar B$, and
it also preserves $\bar q$ because for singular generators
$$
  \bar q\!\left(\sum_i\alpha_i u_i\right)
  =\sum_{i<j}\alpha_i\alpha_j\bar B(u_i,u_j).
$$
So it is an isometry.  Conversely, an isometry clearly preserves all pairings.  The action of
the orthogonal group on spanning tuples is free, because an isometry fixing all $u_i$ is the
identity.  Hence the number of graphs of each type is at most the number of spanning singular
$n$-tuples divided by the order of the corresponding orthogonal group, and this is at most the
number of all singular $n$-tuples divided by that order.

Write $r=2h$.  There are two isometry classes of non-degenerate quadratic forms in even
dimension $r$ over $\F_2$.  The numbers of singular vectors, including $0$, are
$$
  S_\pm(r)=2^{r-1}\pm 2^{h-1}
  \le 2^{r-1}(1+2^{-h})=2^{r-1}(1+2^{-r/2}).
$$
The standard orthogonal-group formulas over $\F_2$~\cite{Taylor1992} give
$$
  |O^\pm(2h,2)|
  =2\,2^{h(h-1)}(2^h\mp1)\prod_{i=1}^{h-1}(2^{2i}-1).
$$
Equivalently,
$$
  |O^+(2h,2)|
  =2^{r(r-1)/2}\,2(1-2^{-h})\prod_{i=1}^{h-1}(1-2^{-2i}),
$$
with the plus sign changed to $1+2^{-h}$ for $O^-(2h,2)$.  In particular
$|O^\pm(r,2)|\ge 2^{r(r-1)/2}$ for every $r\ge2$; for $h=1$ this is immediate, while for
$h\ge2$ the plus-type prefactor is at least
$$
  2(1-2^{-h})\prod_{i=1}^{h-1}(1-2^{-2i})
  \ge 2\cdot\frac34\left(1-\sum_{i\ge1}2^{-2i}\right)=1,
$$
and the minus type is larger.
Therefore
$$
\begin{aligned}
  N_{n,r}
  &\le \sum_{\sigma\in\{+,-\}} \frac{S_\sigma(r)^n}{|O^\sigma(r,2)|}  \\
  &\le 2\,\frac{\big(2^{r-1}(1+2^{-r/2})\big)^n}{2^{r(r-1)/2}}\\
  &=2^{\,n(r-1)+n\logtwo(1+2^{-r/2})-r(r-1)/2+1},
\end{aligned}
$$
as claimed.
\end{proof}

\begin{proposition}\label{preupper}
Let $s_n\in\{0,1,\dots,n\}$ be any integer sequence such that
$\binom{s_n+1}{2}-n\longrightarrow -\infty.$
Then
$d_n(n-s_n)\longrightarrow0.$
In particular, for every fixed $a\in[0,\sqrt2)$,
$d_n\!\left(n-\lfloor a\sqrt n\rfloor\right)\longrightarrow0.$
\end{proposition}

\begin{proof}
Write $s=s_n$, $t=n-s$, and
$\delta_n:=n-\binom{s+1}{2}\longrightarrow+\infty.$
By the $L^2$--TV bound and the eigenvalue dichotomy above,
$$
  d_n(t)^2
  \le \frac14\sum_{\substack{2\le r\le n\\ r\text{ even}}} N_{n,r}\,2^{-rt}.
$$
Using Lemma~\ref{nonzero-count}, the exponent of the $r$-th summand is at most
$$
  F_r(n)
  := -n+rs-\frac{r(r-1)}2+n\logtwo(1+2^{-r/2})+1.
$$
We show that the sum of $2^{F_r(n)}$ tends to $0$.

Let $R_n:=6\lceil\logtwo n\rceil$ and let
$L_r:=\logtwo(1+2^{-r/2})$.  Since $L_r\le L_2=\logtwo(3/2)<1$ and
$s\le\sqrt{2n}$ for all large $n$ (because $\delta_n>0$ eventually), the range
$2\le r\le R_n$ contributes at most
$$
  R_n\,2^{-(1-L_2)n+R_n\sqrt{2n}+1}=o(1).
$$
For $r>R_n$, we have $nL_r\le Cn2^{-r/2}=o(1)$ uniformly in this range.  Moreover
$$
  rs-\frac{r(r-1)}2
  =\binom{s+1}{2}-\frac{(r-s)(r-s-1)}2.
$$
The last quadratic term is non-negative for every integer $r$.  Hence, uniformly for
$r>R_n$,
$$
  F_r(n)
  \le
  -\delta_n-\frac{(r-s)(r-s-1)}2+2
$$
for all sufficiently large $n$.  Therefore
$$
  \sum_{\substack{r>R_n\\ r\text{ even}}}2^{F_r(n)}
  \le
  4\,2^{-\delta_n}\sum_{k\in\mathbb Z}2^{-k(k-1)/2}=o(1).
$$
Combining the two ranges gives $d_n(n-s_n)\to0$.
The fixed-$a$ statement follows because
$\binom{\lfloor a\sqrt n\rfloor+1}{2}-n=(a^2/2-1)n+O(\sqrt n)\to-\infty$ when
$a<\sqrt2$.
\end{proof}

\begin{lemma}
\label{alt-count}
For even $r\in\{0,2,\dots,n\}$, the number of alternating bilinear forms on $\F_2^n$ of rank $r$
(equivalently, symmetric zero-diagonal $n\times n$ $\F_2$-matrices of rank $r$) is at most
$2^{\,r(n-r)+r+\binom{r}{2}}.$
\end{lemma}

\begin{proof}
Let $V=\F_2^n$. An alternating form $B$ of rank $r$ has radical $\mathrm{rad}(B)$ of dimension $n-r$.
Choose the radical subspace $R=\mathrm{rad}(B)$: the number of $(n-r)$-dimensional subspaces of $V$
is the Gaussian binomial coefficient $\binom{n}{n-r}_2=\binom{n}{r}_2$.
Using the product formula,
$$
\binom{n}{r}_2
=\prod_{i=0}^{r-1}\frac{2^{n-i}-1}{2^{r-i}-1}
\le \prod_{i=0}^{r-1}\frac{2^{n-i}}{2^{r-i-1}}
=2^{\,r(n-r)+r}.
$$
Having fixed $R$, the form descends to an alternating form on $V/R\cong \F_2^r$.
The space of alternating bilinear forms on $\F_2^r$ has dimension $\binom{r}{2}$, hence at most $2^{\binom{r}{2}}$ choices.
Multiplying yields the claim.
\end{proof}

\begin{remark}
The overcount comes from replacing the Gaussian binomial coefficient by the crude bound $2^{r(n-r)+r}$ and from counting all alternating forms on $V\backslash R$, rather than only the non-degenerate ones.
\end{remark}

\begin{proposition}
\label{uppertail}
For all $n$ and all integers $c\ge0$,
$d_n(n+c)\ \le\ C_\star\,2^{-c}.$
\end{proposition}

\begin{proof}
By~\eqref{L2bound} and Lemma~\ref{qwbound},
$$
  d_n(t)\le\frac12\sqrt{\sum_{A\neq 0} 2^{-r(A)t}}.
$$
Since $r(A)$ is the rank of an alternating form on $\F_2^n$, it is always even.
Group by rank $r=r(A)$ and apply Lemma~\ref{alt-count}:
$$
  \sum_{A\neq 0}2^{-r(A)t}
  \le
  \sum_{\substack{r\text{ even}\\r\ge 2}}
  2^{\,r(n-r)+r+\binom{r}{2}}\cdot 2^{-rt}.
$$
Let $t=n+c$.  The exponent of $2$ in the $r$-th term is
\begin{align*}
  r(n-r)+r+\binom{r}{2}-r(n+c)
  &= rn-r^2+r+\tfrac{r(r-1)}{2}-rn-rc\\
  &= -r^2+r+\tfrac{r^2-r}{2}-rc\\
  &= -\tfrac{r^2}{2}+\tfrac{r}{2}-rc.
\end{align*}
Hence
$$
  \sum_{A\neq 0}2^{-r(A)(n+c)}
  \le
  \sum_{\substack{r\text{ even}\\r\ge 2}}
  2^{-r^2/2+r/2-rc}.
$$
For $r\ge 2$ and $c\ge 0$ the exponent $-r^2/2+r/2-rc\le -r^2/2+r/2\le -1$.
So,
$$
  \sum_{A\neq 0}2^{-r(A)(n+c)}
  \le
  2^{-2c}\sum_{\substack{r\text{ even}\\r\ge 2}}2^{-r^2/2+r/2}
  =:C_0\,2^{-2c},
$$
where $C_0:=\sum_{r\ge 2,\,2|r}2^{-r(r-1)/2}=\sum_{j\ge1}2^{-j(2j-1)}<\infty$.
Therefore $d_n(n+c)\le \frac12\sqrt{C_0}\,2^{-c}=C_\star\,2^{-c}$.
\end{proof}

\begin{proposition}\label{prop:spectral-gap}
The inversion walk has absolute spectral gap $1-\max_{A\ne0}|\lambda_A|=\tfrac12$.
\end{proposition}

\begin{proof}
For every non-zero $A$, the polar form $B_{q_A}$ is a non-zero alternating form and therefore
has rank at least $2$.  Lemma~\ref{qwbound} gives $|\lambda_A|\le 2^{-1}$.  Equality is
attained when $A$ is a single edge, say $A=\{\{1,2\}\}$, because then
$q_A(x)=x_1x_2$ and $\lambda_A=\E_{x\in\F_2^n}(-1)^{x_1x_2}=\frac{3-1}{4}=\frac12$.
Thus $\max_{A\ne0}|\lambda_A|=1/2$.
\end{proof}
\section{Lower tail: inversion balls and a sharp rank-counting inequality}
\label{s:balls}

Let $\Sym_n(\F_2)$ denote the set of all symmetric $n\times n$ matrices over $\F_2$.
We have $|\Sym_n(\F_2)|=2^{m+n}$ since such a matrix has $m=\binom{n}{2}$ strictly
upper-triangular entries and $n$ diagonal entries.

\begin{lemma}\label{rktail}
Let $M$ be uniformly random in $\Sym_n(\F_2)$. For every integer $s\in\{0,1,\dots,n\}$,
$$
  \Pr\big(\rank_{\F_2}(M)\le n-s\big)\ \le\ \kappa\,2^{-\binom{s+1}{2}}.
$$
\end{lemma}

\begin{proof}
It is convenient to view a symmetric matrix as a symmetric bilinear form $B$ on
$V=\F_2^n$.  For $0\le r\le n$, we count forms of rank at most $r$.
If $\rank(B)\le r$, then its radical contains at least one subspace $R\le V$ of dimension
$n-r$.  We first choose such an $R$ and then overcount the possible forms whose radical
contains $R$.

The number of choices for $R$ is the Gaussian binomial coefficient $\binom{n}{r}_2$.
Once $R$ is fixed, the condition $R\subseteq\rad(B)$ is equivalent to saying that $B$
descends to a symmetric bilinear form on the quotient $V/R$, which has dimension $r$.
The number of symmetric bilinear forms on an $r$-dimensional $\F_2$-space is
$2^{r(r+1)/2}$.  Hence
$$
  \big|\{M\in\Sym_n(\F_2):\rank(M)\le r\}\big|
  \le \binom{n}{r}_2\,2^{r(r+1)/2}.
$$
Using
$$
  \binom{n}{r}_2
  =2^{r(n-r)}\prod_{i=0}^{r-1}\frac{1-2^{i-n}}{1-2^{i-r}}
  \le \kappa\,2^{r(n-r)},
$$
by the definition of $\kappa$, we get
$$
  \Pr(\rank(M)\le r)
  \le \kappa\,2^{r(n-r)+r(r+1)/2-n(n+1)/2}.
$$
Substituting $r=n-s$ gives
$$
  r(n-r)+\frac{r(r+1)}2-\frac{n(n+1)}2=-\frac{s(s+1)}2=-\binom{s+1}{2},
$$
which proves the claim.
\end{proof}

The following proposition translates inversion-ball volume into a statement about low-rank
symmetric matrices, using the group-algebra structure.

\begin{proposition}
\label{ballvol}
Fix a tournament $T_0$ on $[n]$ and write $m=\binom{n}{2}$.
For every integer $s\in\{0,1,\dots,n\}$,
$$
  |B_{n-s}(T_0)|\ \le\ \kappa\,2^{\,m+n-\binom{s+1}{2}}.
$$
Consequently, for the uniform distribution $\pi$ on tournaments,
$$
  \pi(B_{n-s}(T_0))\ \le\ \kappa\,2^{\,n-\binom{s+1}{2}}.
$$
\end{proposition}

\begin{proof}
We embed tournaments into $\Sym_n(\F_2)$ and use the rank-subadditivity of sums of
rank-$1$ matrices.

Choose any symmetric matrix $M_{T_0}\in\Sym_n(\F_2)$ whose strict upper-triangular part
encodes $T_0$ (the diagonal of $M_{T_0}$ is arbitrary; fix it once and for all, say as zero).
For any subset $X\subseteq[n]$, the matrix
$$
  M_X:=\mathbf{1}_X\mathbf{1}_X^\top\in\Sym_n(\F_2)
$$
is the all-ones matrix on the $X\times X$ block and zero elsewhere, and has rank at most $1$.

Let $T\in B_{n-s}(T_0)$.
By definition, there exists a sequence $X_1,\dots,X_t$ with $t\le n-s$ such that
$z(T)=z(T_0)+\sum_{i=1}^t v_{X_i}$ in $G=\F_2^m$. Fix one such sequence $X_1,\dots,X_t$ for each $T\in B_{n-s}(T_0)$.
All objects below are defined with respect to this choice.
Consider the symmetric matrix
$$
  D:=\sum_{i=1}^t M_{X_i}\;\in\Sym_n(\F_2).
$$
The strict upper-triangular part of $D$ coincides with $v_{X_1}+\cdots+v_{X_t}=z(T)-z(T_0)$,
which is the strict upper-triangular part of $M_T-M_{T_0}$.
We therefore \emph{define} the diagonal of $M_T$ so that $M_T-M_{T_0}=D$, i.e.,
$$
  \mathrm{diag}(M_T):=\mathrm{diag}(M_{T_0})+\mathrm{diag}(D)\pmod{2}.
$$
This choice of diagonal is unique given $D$ and $M_{T_0}$, and it ensures
$M_{T_0}+M_T=D$ in $\Sym_n(\F_2)$.

By subadditivity of rank and $\rank(M_{X_i})\le 1$,
$$
  \rank(M_{T_0}+M_T)=\rank(D)\le\sum_{i=1}^t\rank(M_{X_i})\le t\le n-s.
$$

The map $T\mapsto M_T$ defined above is injective: two tournaments $T\neq T'$ differ on some
edge $\{a,b\}$, so the $(a,b)$-entry of $M_T$ and $M_{T'}$ differ, hence $M_T\neq M_{T'}$.
Consequently, $T\mapsto M_{T_0}+M_T$ is also injective (since $M_{T_0}$ is fixed), and the
image is contained in $\{N\in\Sym_n(\F_2):\rank(N)\le n-s\}$.

Therefore
$$
\begin{aligned}
  |B_{n-s}(T_0)|
  &\le\big|\{N\in\Sym_n(\F_2):\rank(N)\le n-s\}\big|\\
  &=|\Sym_n(\F_2)|\cdot\Pr(\rank(M)\le n-s)\\
  &\le \kappa\,2^{m+n-\binom{s+1}{2}},
\end{aligned}
$$
using Lemma~\ref{rktail} with $M$ uniform on $\Sym_n(\F_2)$.
Dividing by the number of tournaments $2^m$ gives the bound on $\pi$.
\end{proof}

\begin{corollary}
\label{lowertail}
For the inversion walk $W_n$ and every integer $s\in\{0,1,\dots,n\}$,
$$
  d_n(n-s)\ \ge\ 1-\kappa\,2^{\,n-\binom{s+1}{2}}.
$$
In particular, for every fixed $a>\sqrt2$,
$d_n\!\left(n-\lfloor a\sqrt n\rfloor\right)\longrightarrow1.$
More generally, for every integer sequence $s_n$ such that
$\binom{s_n+1}{2}-n\longrightarrow +\infty$,
one has
$d_n(n-s_n)\longrightarrow 1.$
For fixed $\eta\in(0,1)$, the explicit lower-tail guarantee $d_n(n-s)\ge 1-\eta$ holds whenever
$$
  s\ge \left\lceil\frac{-1+\sqrt{1+8(n+\log_2(\kappa/\eta))}}2\right\rceil.
$$
\end{corollary}

\begin{proof}
Fix any initial tournament $T_0$.
After $t=n-s$ steps the chain is supported on $B_t(T_0)$, so
$\mu_t(B_t(T_0))=1$, and therefore
$$
  \|\mu_t-\pi\|_{\TV}\ge\mu_t(B_t(T_0))-\pi(B_t(T_0))=1-\pi(B_{n-s}(T_0)).
$$
The first claim follows from Proposition~\ref{ballvol}.  If
$\binom{s_n+1}{2}-n\to+\infty$, then $\kappa\,2^{n-\binom{s_n+1}{2}}\to0$, giving the
second claim.  The last display is obtained by solving
$\kappa\,2^{n-\binom{s+1}{2}}\le\eta$ for $s$.
\end{proof}

\section{Proof of the triangular cutoff threshold}
\label{s:pf-cutoff}

\begin{definition}[Total-variation cutoff window; {\cite{AldousDiaconis1986,Diaconis1996}}]
A sequence of Markov chains with worst-case distances $d_n(t)$ has cutoff at times $t_n$
with window $w_n=o(t_n)$ if, for every fixed $\delta>0$,
$$
  \lim_{n\to\infty} d_n(t_n-\delta w_n)=1,
  \qquad
  \lim_{n\to\infty} d_n(t_n+\delta w_n)=0.
$$
\end{definition}
For this discrete-time chain, non-integer times in a window statement are rounded to the
nearest integer; this changes the displayed offsets by at most $O(1)$ and does not affect any
of the limits below.

\begin{proof}[Proof of Theorem~\ref{mainthm}]
Item~\textup{(i)} is Corollary~\ref{lowertail}, item~\textup{(ii)} is
Proposition~\ref{preupper}, and item~\textup{(iii)} is Proposition~\ref{uppertail}.

Taking $s_n=\lfloor a\sqrt n\rfloor$ gives
$$
  \binom{s_n+1}{2}-n=(a^2/2-1)n+O(\sqrt n),
$$
which tends to $+\infty$ for $a>\sqrt2$ and to $-\infty$ for $0\le a<\sqrt2$.  This proves
the displayed fixed-$a$ transition.

Now let
$\sigma_n:=\max\{s:\ \binom{s+1}{2}\le n\}.$
For $s=\sigma_n-1$ we have
$$
  \binom{s+1}{2}-n=\binom{\sigma_n}{2}-n\le -\sigma_n\to-\infty,
$$
so item~\textup{(ii)} gives $d_n(n-\sigma_n+1)\to0$.  Hence
$t_{\mix}^{(n)}(\varepsilon)\le n-\sigma_n+1$ for all sufficiently large $n$.
For $s=\sigma_n+2$, the definition of $\sigma_n$ implies
$n\le\binom{\sigma_n+2}{2}-1$, and therefore
$$
  \binom{s+1}{2}-n
  =\binom{\sigma_n+3}{2}-n
  \ge \sigma_n+2\to+\infty.
$$
By item~\textup{(i)}, $d_n(n-\sigma_n-2)\to1$, so
$t_{\mix}^{(n)}(\varepsilon)\ge n-\sigma_n-1$ for all sufficiently large $n$.
This proves the stated two-sided localization.  Since
$\sigma_n=\lfloor(\sqrt{8n+1}-1)/2\rfloor=\sqrt{2n}+O(1)$, it also gives
$t_{\mix}^{(n)}(\varepsilon)=n-\sqrt{2n}+O(1)$.

It remains only to justify the final cutoff-window sentence.  Let $w_n\to\infty$ with
$w_n=o(n)$ and fix $\delta>0$.  Let $q_n=\lfloor\delta w_n\rfloor$.  At the earlier time
$n-\sigma_n-q_n$, the corresponding parameter is $s=\sigma_n+q_n$, and
$\binom{s+1}{2}-n\to+\infty$; hence item~\textup{(i)} gives distance tending to $1$.
At the later time $n-\sigma_n+q_n$, either $q_n>\sigma_n$, in which case this time is at
least $n$ and the conclusion follows from $d_n(n)\to0$ and monotonicity, or else
$s=\sigma_n-q_n\ge0$ and
$$
  \binom{s+1}{2}-n
  \le \binom{\sigma_n-q_n+1}{2}-\binom{\sigma_n+1}{2}
  =-q_n\sigma_n+\binom{q_n}{2}\to-\infty.
$$
Then item~\textup{(ii)} gives distance tending to $0$.
\end{proof}

\section{Restricted inversions}
\label{s:restr}

Fix $k\in\{0,1,\dots,n\}$.
The \emph{$k$-restricted inversion walk} $W_{n,k}$ chooses a uniformly random $k$-subset
$X\subseteq[n]$ and inverts $X$.
In the group encoding, $W_{n,k}$ is a Cayley walk on $G=\F_2^m$ driven by
$$
  \mathcal{S}_k:=\{v_X:\ X\subseteq[n],\ |X|=k\},
  \qquad
  H_k:=\langle \mathcal{S}_k\rangle \le G.
$$
Let $V_k$ denote the right-hand side subspace in Theorem~\ref{Hk-thm} (depending on $k\bmod 4$).
The chain is irreducible on cosets of $H_k$, with uniform stationary distribution on each such coset.

\subsection{Parity invariants}
\label{ss:parity}

Identify $G=\F_2^m$ with the family of edge-subsets $F\subseteq\binom{[n]}{2}$ under
symmetric difference.
Define two $\F_2$-linear functionals, the degree-parity map and the edge-count parity:
$$
  \partial:\F_2^m\to\F_2^n,\quad \partial(F)_v:=\deg_F(v)\pmod 2,
  \qquad
  e:\F_2^m\to\F_2,\quad e(F):=|F|\pmod 2.
$$
We compute these on generators.
For a $k$-clique $K_X$ on $X\subseteq[n]$, $\deg_{K_X}(v)=k-1$ if $v\in X$ and $0$
otherwise, so $\partial(K_X)=(k-1)\mathbf{1}_X\pmod 2$.
Hence:
\begin{itemize}
  \item If $k$ is odd, then $k-1$ is even, so $\partial(K_X)=0$ for all $X$; thus
        $H_k\subseteq\ker(\partial)$.
  \item If $k$ is even, then $\partial(K_X)=\mathbf{1}_X\not\equiv 0$; the obstruction
        $\ker(\partial)$ is not automatic.
\end{itemize}
Also, $e(K_X)=\binom{k}{2}\pmod 2$.
Since $\binom{k}{2}=k(k-1)/2$:
\begin{itemize}
  \item $\binom{k}{2}\equiv 0\pmod 2$ iff $k\equiv 0$ or $1\pmod 4$; in these cases
        $e(K_X)=0$ for all $X$ and $H_k\subseteq\ker(e)$.
  \item $\binom{k}{2}\equiv 1\pmod 2$ iff $k\equiv 2$ or $3\pmod 4$; the edge-count
        obstruction does not apply.
\end{itemize}

Combining: $H_k\subseteq\ker(\partial)\cap\ker(e)$ when $k\equiv 1\pmod 4$;
$H_k\subseteq\ker(\partial)$ (but not necessarily $\ker(e)$) when $k\equiv 3\pmod 4$;
$H_k\subseteq\ker(e)$ when $k\equiv 0\pmod 4$; and no obstruction when $k\equiv 2\pmod 4$.
The content of Theorem~\ref{Hk-thm} is that these inclusions are equalities.

\subsection{Boundary cases}
\label{degen}

\paragraph{\texorpdfstring{$k=0$ or $k=1$.}{k = 0 or k = 1.}}
$v_X=0$ for all $X$ with $|X|\le 1$, so $H_k=\{0\}$ and the chain does not move.

\paragraph{\texorpdfstring{$k=n$.}{k = n.}}
There is only one $n$-subset, namely $[n]$, so $H_n=\langle v_{[n]}\rangle$ has size $2$.
Starting from $T_0$, the chain alternates between $T_0$ and its complete reversal $T_0^*$
(the tournament obtained by reversing all arcs).

\paragraph{\texorpdfstring{$k=n-1$.}{k = n - 1.}}
Assume here that $n\ge3$; for $n\le2$ this overlaps with the already degenerate cases
$k\le1$.  The generators are $\{v_{[n]\setminus\{i\}}:i\in[n]\}$.  We claim:
\begin{itemize}
  \item If $n$ is odd: the $n$ generators are linearly independent, so $\dim H_{n-1}=n$.
  \item If $n$ is even: the $n$ generators span a space of dimension $n-1$; there is exactly
        one linear relation, $\sum_{i=1}^n v_{[n]\setminus\{i\}}=0$.
\end{itemize}
\begin{proof}
Write $u_i:=v_{[n]\setminus\{i\}}$. Edge $\{a,b\}$ contributes to $u_i$ iff $i\neq a,b$,
so $\{a,b\}$ appears in $\sum_{i\in S}u_i$ with coefficient $|S\setminus\{a,b\}|\pmod 2$.

For the full sum $\sum_{i=1}^n u_i$: edge $\{a,b\}$ appears $n-2$ times.
Thus $\sum u_i = 0$ in $\F_2^m$ iff $n$ is even; if $n$ is odd the sum is non-zero.

For a proper non-empty sub-sum $\sum_{i\in S}u_i$ with $\emptyset\neq S\subsetneq[n]$:
it suffices to find an edge $\{a,b\}$ for which the coefficient $|S\setminus\{a,b\}|$ is odd.
Such an edge always exists:
if $|S|$ is even, pick $a\in S$ and $b\notin S$, so $|S\setminus\{a,b\}|=|S|-1$ is odd;
if $|S|$ is odd and $|[n]\setminus S|\ge 2$, pick $\{a,b\}\subseteq [n]\setminus S$,
so $|S\setminus\{a,b\}|=|S|$ is odd;
if $|S|$ is odd and $|[n]\setminus S|=1$, pick $\{a,b\}\subseteq S$,
so $|S\setminus\{a,b\}|=|S|-2$ is odd.
Hence some edge has coefficient $1$ and every proper non-empty sub-sum is non-zero.

Hence, for odd $n$ there is no linear relation among the $n$ generators, giving
$\dim H_{n-1}=n$; for even $n$ the unique minimal relation is the full sum, giving
$\dim H_{n-1}=n-1$.
\end{proof}

In both cases the state space of $W_{n,n-1}$ is much smaller than $\F_2^m$ (it has
dimension at most $n$, compared to $m=\binom{n}{2}$).  This boundary case is therefore
qualitatively different from the main regime; as for $k=2$, any convergence statement for the
non-lazy walk must also account for possible period-two obstructions.

\subsection{The main regime \texorpdfstring{$2\le k\le n-2$}{2 <= k <= n - 2}}
\label{ss:main-reg}

The proof of Theorem~\ref{Hk-thm} proceeds by comparing dimensions.
The key tool is Wilson's diagonal form for inclusion matrices.

\begin{lemma}[Wilson~\cite{Wilson1990}; see also Jolliffe~\cite{Jolliffe2024}]
\label{wilson}
Let $W_{2,k}(n)$ be the $0/1$ matrix with rows indexed by $2$-subsets and columns by
$k$-subsets of $[n]$, with entry $1$ iff the $2$-subset is contained in the $k$-subset.
Then over $\F_2$,
$$
  \rank_{\F_2} W_{2,k}(n)
  =
  \sum_{\substack{j\in\{0,1,2\}\\2\nmid \binom{k-j}{2-j}}}
  \left(\binom{n}{j}-\binom{n}{j-1}\right),
  \qquad
  \binom{n}{-1}:=0.
$$
\end{lemma}

Note that the column span of $W_{2,k}(n)$ over $\F_2$ is exactly $H_k$,
since each column is $v_X$ for some $k$-subset $X$, and the columns generate $H_k$.
Hence $\dim H_k=\rank_{\F_2} W_{2,k}(n)$.

We now evaluate the rank formula by case, determining which $j\in\{0,1,2\}$ contribute.

\begin{proof}[Proof of Theorem~\ref{Hk-thm}]

\noindent\textbf{Upper bound $H_k\subseteq V_k$.}
This was established in Section~\ref{ss:parity}.

We check which of the three terms $j=0,1,2$ contribute to the rank sum.

\begin{itemize}
\item \textbf{$j=2$:} $\binom{k-2}{0}=1$, which is odd.
  Contribution: $\binom{n}{2}-\binom{n}{1}=m-n$.

\item \textbf{$j=1$:} $\binom{k-1}{1}=k-1$, which is odd iff $k$ is even.
  Contribution when $k$ is even: $\binom{n}{1}-\binom{n}{0}=n-1$.

\item \textbf{$j=0$:} $\binom{k}{2}=k(k-1)/2$, which is odd iff $k\equiv 2$ or $3\pmod{4}$.
  Contribution: $\binom{n}{0}-0=1$.
\end{itemize}

Summing over the active terms:
$$
  \dim H_k=\rank_{\F_2} W_{2,k}(n)=
  \begin{cases}
  (m-n)+(n-1)+1=m  & k\equiv 2\pmod 4,\\
  m-n+1      & k\equiv 3\pmod 4,\\
  m-1        & k\equiv 0\pmod 4,\\
  m-n        & k\equiv 1\pmod 4.
  \end{cases}
$$

\medskip\noindent\textbf{Dimensions of $V_k$.}
The map $\partial:\F_2^m\to\F_2^n$ is the mod-$2$ vertex-edge incidence map of $K_n$.
Since $K_n$ is connected, the rank of this map (over $\F_2$) equals $n-1$ (Indeed, $\mathrm{Im}(\partial)=\{x\in\F_2^n:\langle x,\mathbf 1\rangle=0\}$ for connected graphs,
so $\mathrm{rank}(\partial)=n-1$).
Hence $\dim\ker(\partial)=m-(n-1)=m-n+1$.

The map $e:\F_2^m\to\F_2$ is non-trivial (a single edge has $e=1$), so $\dim\ker(e)=m-1$.

For the intersection: we need $e|_{\ker(\partial)}\neq 0$, i.e., there exists an edge-set
$F$ with $\partial(F)=0$ (all degrees even) and $|F|$ odd.
A triangle $\{a,b\},\{b,c\},\{a,c\}$ has all vertices of degree $2\equiv 0\pmod 2$ and
has $|F|=3$ (odd), so it lies in $\ker(\partial)\setminus\ker(e)$.
Therefore $e|_{\ker(\partial)}\neq 0$, and $\dim(\ker(\partial)\cap\ker(e))=m-n$.

The dimensions of $V_k$ and $H_k$ match in every case. Since $H_k\subseteq V_k$,
we conclude $H_k=V_k$.
\end{proof}

\subsection{The case \texorpdfstring{$k=2$}{k = 2}: the hypercube}
\label{ss:k2}

When $k=2$, each step flips exactly one uniformly random edge.
In the encoding $z(T)\in\F_2^m$, this is the simple random walk on the $m$-dimensional
hypercube $\{0,1\}^m$.  The non-lazy version has period $2$; we state mixing for the lazy
version $W_{n,2}^L$ (which remains at the current state with probability $1/2$).

\begin{proposition}[Hypercube mixing; {\cite[Ch.~18]{LevinPeres2017}}]
\label{hypercube}
The lazy $k=2$ walk $W_{n,2}^L$ has
$t_{\mix}(\varepsilon)=\frac{m}{2}\ln m+O(m)$ for every fixed $\varepsilon\in(0,1)$, and undergoes cutoff at time $\frac{m}{2}\ln m=\frac{1}{2}\binom{n}{2}\ln\binom{n}{2}$
with window $\Theta(m)=\Theta(n^2)$.
\end{proposition}

Comparing with Theorem~\ref{mainthm}: the full inversion walk mixes at time $\Theta(n)$, a
substantial polynomial-factor improvement over the $\Theta(n^2\log n)$ mixing time for $k=2$.
This reflects the fact that a typical inversion by a $\Theta(n)$-sized set flips $\Theta(n^2)$
edges simultaneously, achieving in one step what the hypercube walk needs $\Theta(n^2)$
steps to accomplish.

\section{Discussion and open problems}
\label{s:open}

Theorem~\ref{mainthm} locates the cutoff for the inversion walk on tournaments at the
integer triangular threshold $n-\sigma_n+O(1)$, where
$\sigma_n=\max\{s:\binom{s+1}{2}\le n\}$.  This answers the mixing-time question posed
in~\cite[Section~8]{AlonPowierskiSaveryScottWilmer2024} with a bounded-window localization.
Theorem~\ref{Hk-thm} gives a complete structural description of the state space of the
$k$-restricted walk for $2\le k\le n-2$.

We collect several natural follow-up directions.

\begin{enumerate}[leftmargin=1.8em, label=(\arabic*)]

\item \textbf{The bounded central profile.}
Theorem~\ref{mainthm} gives the full triangular threshold away from the bounded central
regime: $d_n(n-s_n)\to1$ if $\binom{s_n+1}{2}-n\to+\infty$, and $d_n(n-s_n)\to0$ if
$\binom{s_n+1}{2}-n\to-\infty$.  The remaining natural question is the profile when
$\binom{s_n+1}{2}-n=O(1)$.  In that regime one should expect either a finite family of
subsequential limits, depending on the distance from $n$ to the nearest triangular number, or
a proof that no non-trivial limiting profile exists.

\item \textbf{Mixing of the restricted walk.}
For each fixed $k$, determine the mixing time and cutoff behaviour of $W_{n,k}$ on its state
space.

\item \textbf{Other digraph families.}
The inversion walk and its variants can be defined for general digraphs, not just tournaments.
For digraphs with repeated arcs or for orientations of non-complete graphs, the algebraic
structure is similar but the rank calculations may differ.  Extending the cutoff result to
these settings is a natural generalisation.

\end{enumerate}

\end{document}